\documentclass[12pt]{article} % for a short document

\usepackage[margin=2.4cm]{geometry}
\usepackage[english]{babel}
\usepackage{hyperref}
\usepackage{amsmath}
\usepackage{mathtools}
\usepackage{amsthm}
\usepackage{amssymb}
\usepackage{enumerate}
\usepackage[utf8]{inputenc}
\usepackage[T1]{fontenc}
 \usepackage[affil-it]{authblk}
 
\usepackage{tikz}
\usetikzlibrary{arrows}
\usetikzlibrary{decorations.markings}
\pgfdeclarelayer{nodelayer}
\pgfdeclarelayer{edgelayer}
\pgfsetlayers{nodelayer,main,edgelayer}
\tikzstyle{none}=[draw=none]   
\tikzstyle{bigtiparrow}=[->,thick, >=angle 90]
\tikzstyle{bigtiparrow2}=[->,thick, >=angle 90,preaction={draw=white, -,line width=6pt}]
\tikzstyle{lrarrow}=[<->,thick, >=angle 90,preaction={draw=white, -,line width=6pt}]

\tikzstyle{new}=[rectangle,fill=white,draw=white, inner sep=2pt]
\tikzstyle{new2}=[rectangle,fill=white,draw=white, inner sep=6pt]

\DeclareMathOperator{\supp}{\mathrm{supp}}

\newcommand{\MAlg}{\mathrm{MAlg}}
\newcommand{\Aut}{\mathrm{Aut}}

  \newcommand{\R}{\mathbb R}
  \newcommand{\N}{\mathbb N}

  \newcommand{\inv}{^{-1}}

  \renewcommand{\leq}{\leqslant}

  \newcommand{\norm}[1]{\left\lVert #1\right\rVert}

  \newcommand{\impl}{\Rightarrow}

  \newcommand{\onto}{\twoheadrightarrow}

\newtheorem{thm}{Theorem}[section]
\newtheorem{cor}[thm]{Corollary}
\newtheorem{lem}[thm]{Lemma}
\newtheorem{prop}[thm]{Proposition}
\newtheorem*{qu}{Question}

\theoremstyle{definition}
\newtheorem*{claim}{Claim}

\newtheorem{df}[thm]{Definition}
\newtheorem*{rmq}{Remark}
\newtheorem*{ack}{Aknowledgements}

\renewcommand*{\thefootnote}{\fnsymbol{footnote}}

\title{Polish topologies on groups of non-singular transformations}
\author{François Le Maître\footnote{Research partially supported by ANR AGRUME (ANR-17-CE40-0026) and ANR AODynG (19-CE40-0008-01).}}
\begin{document}

\setcounter{footnote}{1}

\maketitle
\renewcommand*{\thefootnote}{\arabic{footnote}}

\begin{abstract}
In this paper, we prove several results concerning Polish group topologies on groups of non-singular transformation. We first prove that the group of measure-preserving transformations of the real line whose support has finite measure carries no Polish group topology. 
 We then characterize the Borel $\sigma$-finite measures $\lambda$ on a standard Borel space for which the group of $\lambda$-preserving transformations has the automatic continuity property.
 We finally show that the natural Polish topology on the group of all non-singular transformations is actually its only Polish group topology. 
\end{abstract}

\section{Introduction}

The study of measure-preserving (or more generally non-singular) transformations on a standard measured space $(Y,\lambda)$ is broadened once one realises that such transformations form a Polish group. Indeed, the Baire category theorem is then available and so the question of generic properties of such transformations arises naturally. 

As a somewhat degenerate case, one may first look at the case where the measure $\lambda$ is completely atomic. Then $\Aut(Y,\lambda)$ only acts by permuting atoms of the same measure and thus splits as a direct product of permutation groups. In the case where all the atoms have the same measure and $\lambda$ is infinite, we get the Polish group $\mathfrak S_\infty$ of permutations of the integers. In this group, the generic permutation has only finite orbits and infinitely many orbits of size $n$ for every $n\in\N$. Such permutations thus form a comeager conjugacy class.

Actually a much stronger property called \textit{ample generics} holds for the Polish group $\mathfrak S_\infty$, and this has several nice consequences as was shown by Kechris and Rosendal \cite{kechrisTurbulenceAmalgamationGeneric2007}, among which the \emph{automatic continuity property}.

\begin{df}
A Polish group $G$ has the \textbf{automatic continuity property} if whenever $\pi:G\to H$ is a group homomorphism taking values in a separable topological group $H$, the homomorphism $\pi$ has to be continuous. 
\end{df}

It is well-known that as soon as $\lambda$ has a non-atomic part, the group $\Aut(Y,\lambda)$ fails to have ample generics. However, it was shown by Ben Yaacov, Berenstein and Melleray that when $\lambda$ is a non-atomic \emph{finite} measure, $\Aut(X,\lambda)$ still has the automatic continuity property \cite{benyaacovPolishTopometricGroups2013}. 
Later on Sabok developed a framework to show automatic continuity for automorphism groups of metric structures \cite{sabokAutomaticContinuityIsometry2019}. In particular, he got another proof of automatic continuity for $\Aut(Y,\lambda)$, and then Malicki simplified his approach \cite{malickiConsequencesExistenceAmple2016}. 
We first observe that this framework can also  be applied when $\lambda$ is infinite.

\begin{thm}
Let $(Y,\lambda)$ be a standard Borel space equipped with a non-atomic $\sigma$-finite infinite measure $\lambda$. Then $\Aut(Y,\lambda)$ has the automatic continuity property.
\end{thm}

Note that as a concrete example for the above result, one can take $X$ to be the reals equipped with the Lebesgue measure. 
In general, we can actually characterize when $\Aut(X,\lambda)$ has the automatic continuity as follows, where the $\lambda$\emph{-atomic multiplicity} of a real $r>0$ is the number of atoms whose measure is equal to $r$.

\begin{thm}\label{thm:charaauto}
Let $(Y,\lambda)$ be a standard Borel space equipped with a Borel $\sigma$-finite measure $\lambda$. Then the following are equivalent:
\begin{enumerate}[(i)]
\item $\Aut(Y,\lambda)$ has the automatic continuity property;
\item There are only finitely many positive reals whose $\lambda$-atomic multiplicity belongs to $[2,+\infty[$.
\end{enumerate}
\end{thm}

Let us now consider the group $\Aut^*(Y,\lambda)$ of non-singular transformations of $(Y,\lambda)$, i.e. the group of Borel bijections which preserve $\lambda$-null sets. 
If $\lambda_{at}$ denotes the atomic part of $\lambda$ and $\lambda_{cont}$ denotes the atomless part, we see that $\Aut^*(Y,\lambda)$ splits as a direct product 
$$\Aut^*(Y,\lambda)=\Aut^*(Y,\lambda_{at})\times\Aut^*(Y,\lambda_{cont}).$$
The group $\Aut(Y,\lambda_{at})$ is a permutation group, so it has the automatic continuity and thus we focus on $\Aut(Y,\lambda_{cont})$, assuming that $\lambda_{cont}$ is non-trivial. Observe that $\lambda_{cont}$ is then equivalent to an atomless probability measure, so we may as well assume $\lambda_{cont}$ is a probability measure. We are thus led to ask:

\begin{qu}
Let $(X,\mu)$ be a standard probability space. Does the group $\Aut^*(X,\mu)$ of all non-singular transformations of $(X,\mu)$ have the automatic continuity property ? 
\end{qu}

The main difficulty with this question is that the framework of Sabok is not available for $\Aut^*(X,\mu)$ because it cannot be the automorphism group of a complete homogeneous metric structure as was recently shown by Ben Yaacov \cite{benyaacovRoelckeprecompactPolishGroup2018}. While we cannot answer this question, we still manage to obtain a basic consequence of automatic continuity, namely having a unique Polish group topology.

\begin{thm}\label{thm:uniqueautstar}
Let $(X,\mu)$ be a standard probability space. The group $\Aut^*(X,\mu)$ has a unique Polish group topology, namely the weak topology. 
\end{thm}

The techniques we use to prove the above theorem are quite standard, except for the fact that we use the automatic continuity for $\Aut(X,\mu)$ so as to know that $\Aut(X,\mu)$ is a Borel subgroup  
of $\Aut^*(X,\mu)$ for any Polish group topology on $\Aut^*(X,\mu)$. This trick may be of use for other Polish groups.

Finally, we prove a kind of opposite result by showing that the group of all measure-preserving transformations of the real line which have finite support cannot carry any Polish group topology.\\

The paper is organized in two independent sections. Section \ref{sec: pres} deals with groups of measure-preserving transformations over a $\sigma$-finite space. 
After a preliminary section, we start with the above mentioned absence of Polish group topology on the group of finite support transformations in Section \ref{sec:nopolish}. We then check that $\Aut(Y,\nu)$ has the automatic continuity property in Section \ref{sec:infinite MP has AC}, and  we prove Theorem \ref{thm:charaauto} in Section \ref{sec:pfcharaauto}. Section \ref{sec:nonpres} is finally devoted to the proof of the uniqueness of the Polish group topology of $\Aut^*(X,\mu)$ (Theorem \ref{thm:uniqueautstar}).
\begin{rmq}
	Throughout the paper, we will often neglect what happens on null sets without explicit mention.
\end{rmq}

\begin{ack}
	I would like to thank Bruno Duchesne for useful remarks on a previous version of this paper.
\end{ack}

\section{Groups of transformations preserving a $\sigma$-finite measure}\label{sec: pres}

\subsection{Preliminaries}\label{sec:prelim sigma finite}

A \textbf{standard $\sigma$-finite space} is a standard Borel space equipped with a Borel nonatomic $\sigma$-finite infinite measure measure. All such spaces are isomorphic to $\R$ equipped with the Lebesgue measure, and we fix from now on such a standard $\sigma$-finite space $(Y,\lambda)$. 

The first group we are interested in is denoted by $\Aut(Y,\lambda)$ and consists of all Borel bijections $T:Y\to Y$ which preserve the measure $\lambda$: for all Borel $A\subseteq Y$, we have $\lambda(A)=\lambda(T\inv(A))$. As usual two such bijections are identified if they coincide on a conull set. 

Consider the space $\MAlg_f(Y,\lambda)$ of finite measure Borel subsets of $Y$ equipped with the metric $d_\lambda(A,B):=\lambda(A\bigtriangleup B)$, where we identify $A$ and $B$ if $\lambda(A\bigtriangleup B)=0$. Because $Y$ is standard and $\lambda$ is $\sigma$-finite,  the metric space  $(\MAlg_f(Y,\lambda),d_\lambda)$ is separable. Moreover, the Borel-Cantelli lemma yields that $(\MAlg_f(Y,\lambda),d_\lambda)$ is a complete metric space. 

Now if $(X,\mu)$ is a standard probability space, then every Borel subset has finite measure, and by definition the measure algebra $(\MAlg(X,\mu),d_\mu)$ is defined as its set of Borel subsets up to
measure zero, equipped with the metric $d_\mu(A,B):=\mu(A\bigtriangleup B)$.
If $(Z,\nu)$ is another standard probability space, any isometry between  $(\MAlg(X,\mu),d_\mu)$ and $(\MAlg(Z,\nu),d_\nu)$ sending $\emptyset$ to $\emptyset$ comes from a measure-preserving bijection which is unique up to a null set (see \cite[Sec.~1 (B)]{kechrisGlobalaspectsergodic2010}). Using the $\sigma$-finiteness of $(Y,\nu)$ and the above fact, we easily get the following proposition. 

\begin{prop}\label{prop: Aut(y) as iso group}
	$\Aut(Y,\lambda)$ is equal to the group of isometries of $(\MAlg_f(Y,\lambda),d_\lambda)$ which fix $\emptyset$. 
\end{prop}

The above proposition implies that $\Aut(Y,\lambda)$ is a Polish group as it is a closed subgroup of the isometry group of a separable complete metric space. The corresponding topology is called the weak topology; it is thus defined by $T_n\to T$ iff for all $A\subseteq Y$ of finite measure, one has 
$$\lambda(T_n(A)\bigtriangleup T(A))\to 0$$
Note that since $\lambda(T_n(A))=\lambda(T(A))$, this condition is in turn equivalent to $\lambda(T_n(A)\setminus T(A))\to 0$.

For $T\in\Aut(Y,\lambda)$, we define its \textbf{support} to be the following Borel set, which is only well-defined up to measure zero: 
$$\supp T:=\{y\in Y: T(y)\neq y\}.$$
Note that we have the following relation: for all $S,T\in\Aut(Y,\lambda)$, 
$$\supp(STS\inv)=S(\supp T).$$

\begin{df} The group $\Aut_f(Y,\lambda)$ is the normal subgroup of $\Aut(Y,\lambda)$ consisting of all $T\in\Aut(Y,\lambda)$ such that $\lambda(\supp(T))<+\infty$.
\end{df}

\subsection{Absence of Polish group topology on \texorpdfstring{$\Aut_f(Y,\lambda)$}{Autf(Y,lambda)}}\label{sec:nopolish}

\subsubsection{Non-Polishability}

Our first lemma is well-known, we provide a proof for the reader's convenience. 

\begin{lem}\label{lem: approx support} For all $R>0$, the set of $T\in\Aut(Y,\lambda)$ such that $\lambda(\supp T)\leq R$ is closed. 
\end{lem}
\begin{proof}
	Take $T$ such that $\lambda(\supp T)>R$, then there exists a partition of $\supp T$ in countably many sets of positive measure $(A_i)_{i\in\N}$ such that for all $i\in\N$, we have $\mu(T(A_i)\cap A_i)=0$. 
	By our hypothesis, we may then find $n\in\N$ such that $\lambda(A_1\sqcup\cdots\sqcup A_n)>R$, and up to shrinking each $A_i$ we may furthermore assume $\lambda(A_1\sqcup\cdots\sqcup A_n)<+\infty$.
	
	Let $\epsilon=\frac 1n(\lambda(A_1\sqcup\cdots\sqcup A_n)-R)$.  Now take  $T'\in\Aut(Y,\lambda)$ such that $\lambda(T(A_i)\bigtriangleup T'(A_i))<\epsilon$ for  $i=1,...,n$,  let $B_i:=A_i\setminus T'(A_i)$. By construction we have $\lambda(B_i)>\lambda(A_i)-\epsilon$. Moreover $T'(B_i)$ is disjoint from $B_i$ so each $B_i$ is contained in the support of $T'$, and since they are disjoint we conclude that the support of $T'$ has measure greater than $\lambda(A_1\sqcup\cdots\sqcup A_n)-n\epsilon>R$. 
\end{proof}

\begin{df}A subgroup $H$ of a Polish group $G$ is called \textbf{Polishable} if it admits a Polish group topology which refines the topology of $G$. 
\end{df}

\begin{thm}\label{thm: not polishable}The subgroup $\Aut_f(Y,\lambda)\leq \Aut(Y,\lambda)$ is not Polishable.
\end{thm}
\begin{proof}
	Suppose that $\Aut_f(Y,\lambda)$ is Polishable. Then by definition its Polish group  topology $\tau$ refines the weak topology. For each $n\in\N$, let $$F_n:=\{T\in\Aut_f(Y,\mu): \lambda(\supp T)\leq n\}.$$ By the previous lemma, each $F_n$ is closed in $\Aut_f(Y,\lambda)$. Since $\Aut_f(Y,\lambda)=\bigcup_{n\in\N} F_n$, the Baire category theorem yields that there is $n\in\N$ such that $F_n$ has nonempty interior. Since $\tau$ is second-countable, we deduce that $\Aut_f(Y,\lambda)$ is covered by countably many $F_n$-translates. This means that $\Aut_f(Y,\lambda)$ contains a countable set which is $n$-dense for the metric $d_\lambda$ given by 
	$$d_\lambda(T,T'):=\lambda(\{x\in Y: T(x)\neq T'(x)\}).$$
	Let us explain why this cannot happen.
	
	Fix a Borel set $A\subseteq Y$ of measure $3n$, and identify $A$ with the circle $\mathbb S^1$ equipped with the finite measure $3n\lambda$, where $\lambda$ is the Haar measure on $\mathbb S^1$. Take $z\in\mathbb S^1$ and consider $T_z$ the translation by $z$ in $\mathbb S^1$, which we can see through our identification  as a measure preserving transformation of $(Y,\lambda)$ supported on $A$. Observe that for $z\neq z'$, we have $d_\lambda(T_z,T_{z'})=3n$. So in $\Aut_f(Y,\lambda)$ there is an uncountable subgroup all whose distinct elements are $3n$-apart for the metric $d_\lambda$, contradicting the fact that  $\Aut_f(Y,\lambda)$ contains a countable set which is $n$-dense for the metric $d_\lambda$ by the pigeonhole principle.
\end{proof}

\subsubsection{Non-existence of a Polish group topology}

We now upgrade the previous theorem to see that $G\coloneqq\Aut_f(Y,\lambda)$ cannot carry a Polish group topology. Fortunately, the arguments we need were carried out by Kallman in \cite{kallmanUniquenessResultsGroups1985} to prove the uniqueness of the Polish topology of $\Aut(Y,\lambda)$. We reproduce them here for the convenience of the reader. 

For a Borel subset $A\subseteq Y$ we let $$G_{A}:=\{T\in \Aut_f(Y,\lambda): \supp T\subseteq A\}.$$
For a subset $F\subseteq \Aut_f(Y,\lambda)$ we let $$\mathcal C(F):=\{U\in\Aut_f(Y,\lambda): TU=UT \text{ for all }T\in F\}$$ denote its centraliser. 

\begin{lem}\label{lem: centralizer pres}We have $\mathcal C(G_A)=G_{Y\setminus A}$. 
\end{lem}
\begin{proof}
	We clearly have $G_{Y\setminus A}\leq \mathcal C(G_A)$.
	
	Take $T\not\in G_{Y\setminus A}$. Then there exists $B\subseteq A$ with $T(B)$ disjoint from $B$. But clearly $T$ does not commute with an element of $\Aut_f(Y,\lambda)$ supported in $B$, in particular $T\not \in C(G_A)$.
\end{proof}
By this lemma, whenever $\tau$ is a Hausdorff group topology on $\Aut_f(Y,\lambda)$ the set $G_{Y\setminus A}$ is $\tau$-closed. 
Moreover for all $T\in\Aut_f(Y,\lambda)$ and all $A\subseteq Y$, we have \begin{equation}\label{eq:conjugacy}
G_{T(A)}=TG_AT\inv.
\end{equation}
Denote by $G(A,B)$ the set of $T\in\Aut_f(Y,\lambda)$ such that $T(A)\subseteq B$. 
\begin{lem}\label{lem:G(A,B) closed pres}
	Let $\tau$ be a Hausdorff group topology on $G=\Aut_f(Y,Y)$.
	For all $A,B\subseteq Y$, the set $G(A,B)$ is $\tau$-closed.
\end{lem}
\begin{proof}
	Observe first that $A\subseteq B$ if and only if $G_A\leq G_B$: the direct implication is clear, conversely if $A$
	is not a subset of $B$ then we find a transformation supported on $A\setminus B$, thus witnessing
	that $G_A\not\leq G_B$.
	By equality (\ref{eq:conjugacy}), we then have $G(A,B)=\{T\in\Aut(Y,\nu): TG_AT\inv \subseteq G_{B}\}$. 
	So by the previous lemma $G(A,B)$ is the set of all $T\in\Aut(Y,\nu)$ such that for all $U\in G_A$, $TUT\inv$ commutes with every element of $G_{X\setminus B}$. This is clearly a $\tau$-closed condition.
\end{proof}

Now take $A\subseteq Y$, let $\epsilon>0$, and pick $B\subseteq Y$ containing $A$ such that $\lambda(B\setminus A)=\epsilon$. 

\begin{lem}$G_{Y\setminus A}\cdot G(A,B)=\{T\in\Aut_f(Y,\lambda): \lambda(T(A)\setminus A)\leq \epsilon\}$. 
\end{lem}
\begin{proof}[Proof of claim]
	Note that $G_{Y\setminus {A}}$ is a group, and that the set $F:=\{T\in\Aut_f(Y,\lambda): \lambda(T(A)\setminus A)\leq \epsilon\}$ is left $G_{Y\setminus {A}}$-invariant. Moreover since $\lambda(B\setminus A)=\epsilon$ we clearly have $\mathbb G(A,B)\subseteq F$ so $G_{Y\setminus A_m}\cdot G(A_n,B)\subseteq F$.
	
	For the reverse inclusion, take $T\in F$. Since $\mu(T(A)\setminus A)\leq \epsilon$ and $\mu(B\setminus A_m)=\epsilon$ we may find $U\in G_{Y\setminus A_m}$ such that $U(T(A)\setminus A)\subseteq B\setminus A$. We conclude that $UT\in G(A,B)$ so $T\in G_{Y\setminus A}\cdot G(A,B)$.
\end{proof}

The above lemma implies that if $\tau$ is a Polish group topology on $\Aut_f(Y,\lambda)$, then for all $A\subseteq Y$ and $\epsilon>0$, the set $\{T\in\Aut_f(Y,\lambda): \lambda(T(A)\setminus A)\leq \epsilon\}$ is analytic (it is the pointwise product of two closed sets) hence Baire-measurable. 
It easily follows that the inclusion map $\Aut_f(Y,\lambda)\to \Aut(Y,\lambda)$ is Baire-measurable, hence continuous by Pettis' lemma (see e.g. \cite[Thm.~2.3.2]{gaoInvariantDescriptiveSet2009}). But this is impossible by Theorem \ref{thm: not polishable}. This proves the following result.

\begin{thm}
	The group $\Aut_f(Y,\lambda)$ cannot carry a Polish group topology.
\end{thm}

\subsection{Automatic continuity for \texorpdfstring{$\Aut(Y,\lambda)$}{Aut(Y,lambda)}}\label{sec:infinite MP has AC}

Let us now briefly indicate why $\Aut(Y,\lambda)$ has the automatic continuity property by checking the criterions given by Sabok \cite{sabokAutomaticContinuityIsometry2019} and then simplified by Malicki \cite{malickiConsequencesExistenceAmple2016}. We won't give full details since the proofs adapt verbatim and we refer the reader to Malicki's paper for definitions of the terms used thereafter. 

As explained in Section \ref{sec:prelim sigma finite}, we may view the group $\Aut(Y,\lambda)$ as the group of automorphisms of the metric structure $(\MAlg_f(Y,\lambda),d_\lambda, \bigtriangleup, \cap)$ where 
\begin{itemize}
	\item $\MAlg_f(Y,\lambda)$ is the set of finite measure Borel subsets of $Y$, up to measure zero;
	\item $d_\lambda(A,B)=\lambda(A\bigtriangleup B)$;
	\item $\bigtriangleup$ and $\cap$ are the usual set-theoretic operations, thus making $(\MAlg_f(Y,\lambda), \bigtriangleup, \cap)$ a Boolean algebra (without unit).
\end{itemize}

%
%
%
%Note that $G$ acts diagonally on $\MAlg_f(Y,\lambda)^n$; accordingly we will denote by $G(A_1,...,A_n)$ the orbit of $(A_1,...,A_n)\in \MAlg_f(Y,\lambda)^n$ under this action.
%Given an uple $(A_1,...,A_n)$ of finite measure subsets of $(Y,\lambda)$, we denote by $G_{(A_1,...,A_n)}$ denote its stabiliser, i.e. the set of $T\in G$ such that $T(A_i)=A_i$ for every $i\in\{1,...,n\}$. 
%
%Given a tuple $(A_1,...,A_n)$ of finite measure subsets of $(Y,\lambda)$, a subset $Y\subseteq G(A_1,...,A_n)$ is called \textbf{relatively saturated} if there is $\epsilon>0$ such that whenever there is $g\in G$ with
%$$\mu(A_i\bigtriangleup gA_i)<\epsilon$$ for all $i=1,...,n$, then there is $h\in G_{A_1,...,A_n}$ such that $hg(A_1,...,A_n)\in Y$. 
%
%Observe that $G(A_1,...,A_n)$ is a closed set because it is the set of tuples $(B_1,...,B_n)$ such that $\lambda(B_i\cap B_j)=\lambda(A_i\cap A_j)$ for all $i,j\in\{1,...,n\}$. A tuple $(A_1,...,A_n)$ is \textbf{ample} if the set of $(B_1,...,B_n)\in G(A_1,...,A_n)$ such that $G_{(B_1,...,B_n)}(A_1,...,A_n)$ is relatively saturated, is comeager in $G(A_1,...,A_n)$.  

It is well-known that $\MAlg_f(Y,\lambda)$ is homogeneous and complete as a metric structure. 

\begin{lem}
	Finite tuples of disjoint subsets of $(Y,\lambda)$ are ample and relevant.
\end{lem}
\begin{proof}
	The proof of \cite[Lem.~6.2]{malickiConsequencesExistenceAmple2016} adapts verbatim.
\end{proof}

\begin{lem}
	$\MAlg_f(Y,\lambda)$ locally has finite automorphisms and has the extension property.
\end{lem}
\begin{proof}
	Every finitely generated substructure of $\MAlg_f(Y,\lambda)$ has a unit $X$ so that we may see it as a substructure of the measure algebra $\MAlg(X,\lambda_X)$, which up to rescaling is the measure algebra over a standard probability space. 
	The result then follows from \cite[Lem.~8.1 and Lem.~8.2]{sabokAutomaticContinuityIsometry2019}.
\end{proof}

As a consequence of Malicki's theorem \cite[Thm.~3.4]{malickiConsequencesExistenceAmple2016}, we thus have the following result.

\begin{thm}\label{thm: AC for infinite measure space}
	The group $\Aut(Y,\lambda)$ of measure-preserving transformation of an infinite $\sigma$-finite standard measured space has the automatic continuity property.
\end{thm}
\begin{cor}[{Kallman \cite{kallmanUniquenessResultsGroups1985}}]
	The group $\Aut(Y,\lambda)$ has a unique Polish group topology.
\end{cor}

\begin{rmq}
	Let $\MAlg_1(Y,\lambda)$ denote the closed set of all $A\in \MAlg_f(Y,\lambda)$ whose measure is at most $1$. 
	It is easy to check that the $\Aut(Y,\lambda)$-action on $\MAlg_1(Y,\lambda)$ is approximately oligomorphic
	and that $\Aut(Y,\lambda)$ is a closed subgroup of the isometry group of $\MAlg_1(Y,\lambda)$. 
	By \cite[Thm.~2.4]{benyaacovWeaklyAlmostPeriodic2016}, we conclude 
	that $\Aut(Y,\lambda)$ is a Roecke precompact Polish group.
\end{rmq}

\subsection{A characterization of automatic continuity}\label{sec:pfcharaauto}

We finally use the previous results to characterize automatic continuity for $\Aut(Y,\lambda)$, where $(Y,\lambda)$ be a standard Borel space equipped with a Borel $\sigma$-finite measure $\lambda$. Recall that for such a measure there are only countably many atoms and they have finite measure (by $\sigma$-finiteness), and that each atom is a singleton (because $Y$ is standard). 
Let us  say that the $\lambda$-\textbf{atomic multiplicity} of a positive real $r$ is the (possibly infinite) number of atoms in $Y$ whose measure is equal to $r$.

\begin{thm}
	Let $(Y,\lambda)$ be a standard Borel space equipped with a Borel $\sigma$-finite measure $\lambda$. Then the following are equivalent:
	\begin{enumerate}[(i)]
		\item $\Aut(X,\lambda)$ has the automatic continuity property;
		\item There are only finitely many positive reals whose $\lambda$-atomic multiplicity belongs to $[2,+\infty[$.
	\end{enumerate}
\end{thm}
\begin{proof}
	We first prove the contrapositive of (i)$\impl$(ii). Suppose there are infinitely many positive reals whose $\lambda$-atomic multiplicity belongs to $[2,+\infty[$ and enumerate them as $(r_n)_{n\in\N}$. Then if $A_n$ is the set of atoms of measure $r_n$, we see that each $A_n$ is $\Aut(Y,\lambda)$-invariant and we thus get natural surjection $$\Aut(Y,\lambda)\onto \prod_n \mathfrak S(A_n).$$
	For each $n$, let $\sigma_n$ be the signature map $\mathfrak S(A_n)\onto\{\pm 1\}$. By composing our previous homomorphism with $(\sigma_n)_{n\in\N}$ we get a continuous surjection $\Aut(Y,\lambda)\onto \{\pm1\}^\N$. Since the latter has $2^{2^{\aleph_0}}$ distinct homomorphisms onto $\{\pm 1\}$ (indeed each ultrafilter on $\N$ provides such a homomorphism) and there are at most $2^{\aleph_0}$ continuous homomorphisms $\Aut(Y,\lambda)\to \{\pm 1\}$, we conclude that $\Aut(Y,\lambda)$ does not have the automatic continuity property. 
	
	We now prove (ii)$\impl$(i). Let $(r_i)_{i=1}^n$ be the reals whose $\lambda$-atomic multiplicity belongs to $[2,+\infty[$ and let $A_i$ be the set of atoms of measure $r_i$. Let $(s_j)_{j\in J}$ denote the reals whose $\lambda$-atomic multiplicity is infinite and let $B_j$ be the set of atoms of measure $s_j$. Finally, let $\eta$ be the non-atomic part of $\lambda$. We then have a decomposition
	\begin{equation}\label{eqn:decomposition}\Aut(Y,\lambda)=\Aut(Y,\eta)\times\prod_{i=1}^n\mathfrak S(A_i)\times\prod_{j\in J}\mathfrak S(B_j),\end{equation}
	where $\mathfrak S(B_j)$ is equipped with the topology of pointwise convergence, viewing $B_j$ as a discrete set. 
	
	Let us  show that $\Aut(Y,\eta)$, $\prod_{i=1}^n\mathfrak S(A_i)$ and $\prod_{j\in J}\mathfrak S(B_j)$ have automatic continuity. Since $\lambda$ is $\sigma$-finite, $\eta$ also is. We then have three cases to check.
	\begin{itemize}
		\item If $\eta$ is trivial, $\Aut(Y,\eta)$ also is and hence has automatic continuity.
		\item If $\eta$ is finite, $\Aut(Y,\eta)$ has automatic continuity by 
		\cite[Thm.~6.2]{benyaacovPolishTopometricGroups2013}. 
		\item If $\eta$ is infinite, $\Aut(Y,\eta)$ has automatic continuity by Theorem \ref{thm: AC for infinite measure space}. 
	\end{itemize}
	The group $\prod_{i=1}^n\mathfrak S(A_i)$ is finite and thus has automatic continuity. Finally the group $\prod_{j\in J}\mathfrak S(B_j)$ is a countable product of groups with ample generics and hence has ample generics. By \cite[Thm.~1.10]{kechrisTurbulenceAmalgamationGeneric2007} it has automatic continuity.
	
	Since any finite product of groups with automatic continuity has automatic continuity, we conclude from \eqref{eqn:decomposition} that $\Aut(Y,\eta)$ has the automatic continuity property.
\end{proof}

\section{The group of non-singular transformations}\label{sec:nonpres}

\subsection{Preliminaries}

A \textbf{standard probability space} is a standard Borel space equipped with a Borel nonatomic probability measure. All such spaces are isomorphic, and we fix from now on such a standard probability space $(X,\mu)$. 

A Borel bijection $T$ of $(X,\mu)$ is called \textbf{non-singular} if the pushforward measure $T_*\mu$ is equivalent to $\mu$, that is, if for all Borel $A\subseteq X$, we have $\mu(A)=0$ if and only if $\mu(T\inv(A))=0$. Denote by $\Aut^*(X,\mu)$ the group of non-singular Borel bijections of $(X,\mu)$, two such bijections being identified if they coincide up to measure zero. 

The \textbf{weak topology} on $\Aut^*(X,\mu)$ is a metrizable group topology defined by declaring that a sequence $(T_n)$ of elements of $\Aut^*(X,\mu)$ weakly converges to  $T\in\Aut^*(X,\mu)$ if for all Borel $A\subseteq X$, one has $\mu(T_n(A)\bigtriangleup T(A))\to 0$ and  
\begin{equation}\norm{ \frac{d(T_{n*}\mu)}{d\mu}-\frac{d(T_*\mu)}{d\mu} }_1\to 0.\end{equation}
We refer the reader to \cite{danilenkoErgodicTheoryNonsingular2011} for more on this topology, which is actually a Polish group topology. 
Our purpose here will be to show that it is the unique Polish group topology one can put on $\Aut^*(X,\mu)$.

For $T\in\Aut^*(X,\mu)$, we define as before its \textbf{support} to be the Borel set 
$$\supp T:=\{x\in X: T(x)\neq x\}.$$
Note that we have again the following relation: for all $S,T\in\Aut^*(X,\mu)$, 
$\supp(STS\inv)=S(\supp T)$.

We denote by $\Aut(X,\mu)$ the group of measure-preserving transformations of $(X,\mu)$, which
is a closed subgroup of $\Aut^*(X,\mu)$.
Similarly to Proposition \ref{prop: Aut(y) as iso group}, $\Aut(X,\mu)$ is the group of isometries of
the \textbf{measure algebra} $\MAlg(X,\mu)$, defined as the set of Borel subsets of $(X,\mu)$ up to measure zero and equipped with the metric $d_\mu(A,B)=\mu(A\bigtriangleup B)$.
The following  theorem of Ben Yaacov, Berenstein and Melleray will imply it is always a Borel subset regardless of the Polish group topology we put on $\Aut^*(X,\mu)$. 

\begin{thm}[{see \cite[Thm.~6.3]{benyaacovPolishTopometricGroups2013}}]Every homomorphism $\Aut(X,\mu)\to H$ where $H$ is a separable topological group has to be continuous. 
\end{thm}

Finally, we will need the following kind of converse of the fact that elements of $\Aut(X,\mu)$ preserve the measure: given two subsets $A,B\subseteq X$ of the same measure, there exists $T\in\Aut(X,\mu)$ supported on $A\cup B$ such that $T(A)=B$ up to measure $0$.

\subsection{Uniqueness of the Polish group topology of $\Aut^*(X,\mu)$}

\begin{thm}The weak topology is the unique Polish group topology on the group $\Aut^*(X,\mu)$.
\end{thm}

\begin{proof}

Let us fix a countable dense subalgebra of $\MAlg(X,\mu)$ and enumerate it as $(A_n)_{n\in\N}$. For $m,n,k\in\N$, we let 
$$\mathbb B_{n,m,k}:=\left\{T\in\Aut^*(X,\mu): \mu(T(A_n)\setminus A_m)\leq \frac 1{2^k}\right\}.$$

Let us first show that the Borel group structure of $\Aut^*(X,\mu)$ is generated by the subsets $\mathbb B_{n,m,k}$. 

First note that since $\Aut^*(X,\mu)$ acts continuously on $\MAlg(X,\mu)$, we have that each $\mathbb B_{n,m,k}$ is closed, hence Borel. 

By density of $(A_n)$ in $\MAlg(X,\mu)$, we have that $(\mathbb B_{n,m,k})_{n,m,k\in\N}$ is a countable \textit{separating} family of Borel subsets of the standard Borel space $\Aut^*(X,\mu)$. We conclude that $(\mathbb B_{n,m,k})_{n,m,k\in\N}$ generates the Borel $\sigma$-algebra of $\Aut^*(X,\mu)$. \\ 

Let now $\tau$ be a Polish group topology on $\mathbb G\coloneqq \Aut^*(X,\mu)$. To conclude that $\tau$ is the weak topology, it suffices to show that each $\mathbb B_{n,m,k}$ is $\tau$-Baire-measurable. We need a few easy lemmas already present in the previous section.

For a Borel subset $A\subseteq X$, we let $\mathbb G_{A}$ denote the group of $T\in\Aut^*(X,\mu)$ such that $\supp T\subseteq A$. For a subset $\mathbb B\subseteq \Aut^*(X,\mu)$, let $\mathcal C(\mathbb B)$ denote its centraliser.
We now repeat the short proofs of Lemma \ref{lem: centralizer pres} and \ref{lem:G(A,B) closed pres}.

\begin{lem}\label{lem: GA closed}For all $A\subseteq X$ we have $\mathcal C(\mathbb G_A)=\mathbb G_{X\setminus A}$. In particular $\mathbb G_A$ is $\tau$-closed.
\end{lem}
\begin{proof}
We clearly have $\mathbb G_{X\setminus A}\leq \mathcal C(\mathbb G_A)$.

Take $T\not\in \mathbb G_{X\setminus A}$. Then there exists $B\subseteq A$ with $T(B)$ disjoint from $B$. But clearly $T$ does not commute with any nontrivial element of $\Aut^*(X,\mu)$ supported in $B$, in particular $T\not \in \mathcal C(\mathbb G_A)$.
\end{proof}

Note that for all $T\in\Aut^*(X,\lambda)$ and all $A\subseteq X$, we have again
$\mathbb G_{T(A)}=T\mathbb G_AT\inv.$
Denote by $\mathbb G(A,B)$ the set of $T\in\Aut^*(X,\mu)$ such that $T(A)\subseteq B$. 

\begin{lem}\label{lem:G(A,B) closed}
For all $A,B\subseteq X$, the set $\mathbb G(A,B)$ is $\tau$-closed.
\end{lem}
\begin{proof}
By the equality (\ref{eq:conjugacy}), we have $\mathbb G(A,B)=\{T\in\Aut^*(X,\mu): T\inv \mathbb G_AT\subseteq \mathbb G_{B}\}$. 
So by the previous lemma $\mathbb G(A,B)$ is the set of all $T\in\Aut^*(X,\mu)$ such that for all $U\in\mathbb G_A$, $TUT\inv$ commutes with every element of $\mathbb G_{X\setminus B}$. This is clearly a $\tau$-closed condition.
\end{proof}

We now make a crucial remark which relies on the automatic continuity property for $\Aut(X,\mu)$. 

\begin{lem}\label{lem: HA borel}
For $A\subseteq X$, let $\mathbb H_A=\{T\in\Aut(X,\mu): \supp T\subseteq A\}$. Then $\mathbb H_A$ is a $\tau$-Borel subset of $\Aut^*(X,\mu)$.
\end{lem}
\begin{proof}
By the automatic continuity property for $\Aut(X,\mu)$, we know that $\Aut(X,\mu)$ has to be a $\tau$-Borel subset of $\Aut^*(X,\mu)$. But $\mathbb H_A=\mathbb G_A\cap \Aut(X,\mu)$ and by Lemma \ref{lem: GA closed} we have that $\mathbb G_A$ is closed, so $\mathbb H_A$ is Borel. 
\end{proof}

\begin{rmq}
	We could replace $\Aut(X,\mu)$ by the full group of any measure-preserving ergodic equivalence relation on $(X,\mu)$ in our argument. 
	Indeed such a group also has the automatic continuity property by a result of Kittrell and Tsankov \cite{kittrellTopologicalPropertiesFull2010}, and acts transitively on sets of equal measure. 
\end{rmq}

Let $n,m,k\in\N$; we finally prove that the set $$\mathbb B_{n,m,k}=\left\{T\in\Aut^*(X,\mu): \mu(T(A_n)\setminus A_m)\leq \frac 1{2^k}\right\}$$ is $\tau$-Baire-measurable. We may assume that $\mu(A_m)<1-\frac 1{2^k}$ because otherwise $\mathbb B_{n,m,k}=\Aut^*(X,\mu)$. Let $B$ a Borel set containing $A_m$ such that $\mu(B)=\mu(A_m)+\frac 1{2^k}$. 

\begin{claim}We have $\mathbb B_{n,m,k}=\mathbb H_{X\setminus {A_m}}\cdot \mathbb G(A_n, B)$
\end{claim}
\begin{proof}[Proof of claim]
Note that $\mathbb H_{X\setminus {A_m}}$ is a group, and that $\mathbb B_{n,m,k}$ is left $\mathbb H_{X\setminus {A_m}}$-invariant. Moreover since $\mu(B\setminus A_m)=\frac 1{2^k}$ we clearly have $\mathbb G(A_n,B)\subseteq\mathbb B_{n,m,k}$ so $\mathbb H_{X\setminus A_m}\cdot\mathbb G(A_n,B)\subseteq\mathbb B_{n,m,k}$.

For the reverse inclusion, take $T\in \mathbb B_{n,m,k}$. Since $\mu(T(A_n)\setminus A_m)\leq \frac 1{2^k}$ and $\mu(B\setminus A_m)=\frac 1{2^k}$ we may find $U\in\mathbb H_{X\setminus A_m}$ such that $U(T(A_n)\setminus A_m)\subseteq B\setminus A_m$. We conclude that $UT\in\mathbb G(A_n,B)$ so $T\in \mathbb H_{X\setminus A_m}\cdot\mathbb G(A_n,B)$.
\end{proof}

By Lemma \ref{lem:G(A,B) closed} the set $\mathbb G(A_n, B)$ is $\tau$-closed, while by Lemma \ref{lem: HA borel} the set $\mathbb H_{X\setminus A_m}$ is $\tau$-Borel. Being the pointwise product of two Borel sets, the set $\mathbb B_{n,m,k}$ is analytic, hence Baire-measurable. 

We can then conclude the proof in a standard manner:
since the sets $\mathbb B_{n,m,k}$ generate the $\sigma$-algebra of the weak topology $w$ on $\Aut^*(X,\mu)$,
the identity map $(\Aut^*(X,\mu),\tau)\to(\Aut^*(X,\mu),w)$ is continuous by Pettis' lemma 
\cite[Thm.~2.3.2]{gaoInvariantDescriptiveSet2009}. Being injective, its inverse is Borel by the Lusin-Suslin theorem \cite[Thm.~15.1]{kechrisClassicaldescriptiveset1995}, and thus continuous as well by one last application of Pettis' lemma.
\end{proof}

%\begin{lem}
%The group $\Aut(X,\lambda)$ is a closed subgroup of $\Aut^*(X,\mu)$; in particular it is also a Polish group for the weak topology.
%\end{lem}
%
%\begin{proof}
%Let $A$ be a Borel subset of $X$ such that $\lambda(A)<+\infty$ and consider the set
%\begin{equation*}
%\Omega_A:=\{T\in\Aut^*(X,\mu): \lambda(T(A))> \lambda(A)\}
%\end{equation*}
%We will show that $\Omega_A$ is open. To this end, let $T\in\Omega_A$. We have that $\lambda(A)<+\infty$ and also that $\lambda(A)<\lambda(T(A))$. That $\lambda$ is absolutely continuous with respect to $\mu$ implies the existence of $\epsilon>0$ such that whenever $B$ is a Borel set satisfying $\mu(T(A)\setminus B)<\epsilon$, we have $\lambda(B)>\lambda(A)$. Then the set
%$\{U\in\Aut^*(X,\mu): \mu(U(A)\bigtriangleup T(A))<\epsilon \}$
%is an open neighborhood of $T$ contained in $\Omega_A$, so $\Omega_A$ is open.
%
%For $T\in\Aut^*(X,\mu)$, it now follows $T\in\Aut(X,\lambda)$ if and only if $T\not\in\Omega_A$ and $T\inv\not\in\Omega_A$ for every Borel $A\subseteq X$ such that $\lambda(A)<+\infty$. As a consequence, $\Aut(X,\lambda)$ is a closed subgroup of $\Aut^*(X,\lambda)$.
%\end{proof}

\bibliographystyle{alpha}
\bibliography{/home/francois/Nextcloud/Maths/zoterobib}

\end{document}